\documentclass{amsart}

\usepackage{amsfonts,amssymb,amsmath,amsthm}
\usepackage{tikz}
\usepackage{tikz-cd}
\usepackage{array}
\usepackage{longtable}

\newtheorem{thm}{Theorem}[section]
\newtheorem{cor}[thm]{Corollary}
\newtheorem{lem}[thm]{Lemma}
\newtheorem{prop}[thm]{Proposition}

\newtheorem{conj}[thm]{Conjecture}

\theoremstyle{definition}
\newtheorem{defn}[thm]{Definition}
\newtheorem{exm}[thm]{Example}

\theoremstyle{remark}
\newtheorem{remark}[thm]{Remark}

\newcommand{\kk}{\mathbb{K}}

\newcommand{\m}{\mathbf{m}}
\newcommand{\n}{\tilde{n}}
\newcommand{\nepsilon}{\tilde{\epsilon}}

\newcommand{\Z}{\mathbb{Z}}
\newcommand{\A}{\mathcal{A}}
\newcommand{\dev}{\mbox{DV}}

\usepackage{hyperref}

\title{Inequalities for free multi-braid arrangements}

\author[M.~DiPasquale]{Michael DiPasquale}     
\address{Michael DiPasquale\\     
	Department of Mathematics\\     
	Oklahoma State University\\     
	Stillwater\\
	OK \ 74078-1058\\     
	USA}     
\email{mdipasq@okstate.edu}
\urladdr{\url{http://math.okstate.edu/people/mdipasq/}}

\begin{document}

\begin{abstract}
We prove that, on a large cone containing the constant multiplicities, the only free multiplicities on the braid arrangement are those identified in work of Abe, Nuida, and Numata (2009).  We also give a conjecture on the structure of all free multiplicities on braid arrangements.
\end{abstract}

\maketitle

\section{Introduction}
Let $V\cong \kk^{\ell+1}$ be a vector space over a field $\kk$ of characteristic zero, $V^*$ its dual space and $S=\mbox{Sym}(V^*)\cong\kk[x_0,\ldots,x_\ell]$.  Given a polynomial $f\in S$ denote by $V(f)$ the zero-locus of $f$ in $V$.  The \textit{braid arrangement of type} $A_{\ell}\subset V$ is defined as $A_\ell=\cup_{0\le i<j\le\ell} H_{ij}$, where $H_{ij}=V(x_i-x_j)$.  A \textit{multiplicity} on $A_\ell$ is a map $\m:\{H_{ij}\}\rightarrow \Z_{>0}$; we will set $m_{ij}=\m(H_{ij})$.  The pair $(A_{\ell},\m)$ is called a multi-arrangement.  The multi-arrangement $(A_\ell,\m)$ is \textit{free} if the corresponding module $D(A_{\ell},\m)$ of multi-derivations (i.e., vector fields tangent to $A_\ell$ with multiplicities prescribed by $\m$) is a free module over the polynomial ring $\kk[x_0,\ldots,x_\ell]$. (See Section~\ref{sec:notation} for more details.)  If $(A_\ell,\m)$ is free, we say $\m$ is a \textit{free multiplicity}.  Free multiplicities on braid arrangements have been studied since the introduction of the module of logarithmic differentials by Saito~\cite{SaitoUniformization}, largely due to their importance in the theory of Coxeter arrangements and later in connection with a conjecture of Athanasiadis~\cite{AthDefCoxeter}.  Terao made a major breakthrough in~\cite{TeraoMultiDer}, showing that the constant multiplicity on any Coxeter arrangement is free and determining the corresponding exponents.  Subsequently, many authors studied freeness of `almost-constant' multiplicities on Coxeter and braid arrangements~\cite{TeraoDoubleCoxeter,TeraoMultiDer,YoshinagaPrimitiveDerivationMultiCoxeter,AbeQuasiConstant}.   In the setting of the braid arrangement, this line of inquiry resulted in a paper of Abe-Nuida-Numata~\cite{AbeSignedEliminable}, where the authors classify what we shall call \textit{ANN multiplicities}.  Given non-negative integers $n_0,\ldots,n_\ell$ and integers $\epsilon_{ij}\in\{-1,0,1\}$ for all $0\le i<j\le \ell$, an ANN multiplicity is a multiplicity satisfying
\begin{enumerate}
\item $m_{ij}=n_i+n_j+\epsilon_{ij}$ and
\item $m_{ij}\le m_{ik}+ m_{jk}+1$ for every triple $i,j,k$.
\end{enumerate}
In~\cite{AbeSignedEliminable} ANN multiplicities are classified as free if and only if a corresponding signed graph is \textit{signed-eliminable}; we will describe this precisely in \S~\ref{sec:SignedEliminable}.  We shall refer to the set of multiplicities satisfying the inequalities in $(2)$ as \textit{the balanced cone} of multiplicities.  The reason for this name will be explained in \S~\ref{sec:AnyBraidNonFree}.

In this note we prove that a multiplicity in the balanced cone is free if and only if it is a free ANN multiplicity.  This partially generalizes the recent classification of all free multiplicities on the $A_3$ braid arrangement~\cite{A3MultiBraid}, which is joint work of the author with Francisco, Mermin, and Schweig.  To state our result more concretely we shall associate to the multi-braid arrangement $(A_\ell,\m)$ an edge-labeled complete graph $(K_{\ell+1},\m)$.  The vertices of $K_{\ell+1}$ are labeled in bijection with the variables $x_0,\ldots,x_\ell\in S$.  An edge $\{v_i,v_j\}$ corresponds to $H_{ij}=V(x_i-x_j)$ and is furthermore labeled by $\m(H_{ij})=m_{ij}$.  Now suppose $C$ is a four-cycle in $K_{\ell+1}$ which traverses the vertices $v_i,v_j,v_s,v_t$ in order.  Define $\m(C)=|m_{ij}-m_{js}+m_{st}-m_{it}|$; since we take absolute value, $\m(C)$ is independent of orientation, depending only on the four cycle and the multiplicity.  Let $C_4(K_{\ell+1})$ be the set of all four cycles of $K_{\ell+1}$.  Given a subset $U\subset\{v_0,\ldots,v_\ell\}$ of size at least four, the \textit{deviation of }$\m$\textit{ over }$U$ is
\[
\dev(\m_U)=\sum_{\substack{C\in C_4(K_{\ell+1})\\ C\subset U}} \m(C)^2.
\]  
Our main result is the following.

\begin{thm}\label{thm:1}
Suppose $(A_\ell,\m)$ is a multi-braid arrangement with $\m$ in the balanced cone of multiplicities.  For a subset $U\subset\{v_0,\ldots,v_\ell\}$ let $\m_U=\{m_{ij}\mid \{v_i,v_j\}\subset U\}$ and denote by $q_U$ the number of integers $\{m_{ij}+m_{ik}+m_{jk}\mid \{v_i,v_j,v_k\}\subset U\}$ that are odd.  Then the following are equivalent.
\begin{enumerate}
\item $(A_\ell,\m)$ is free
\item $\dev(\m_U)\le q_U(|U|-1)$ for every subset $U\subset \{v_0,\ldots,v_\ell\}$ where $|U|\ge 4$.
\item $\m$ is a free ANN multiplicity.  In other words, there exist non-negative integers $n_0,\ldots,n_\ell$ and $\epsilon_{ij}\in\{-1,0,1\}$ (for $0\le i<j\le \ell$) so that 
\begin{enumerate}
\item $m_{ij}=n_i+n_j+\epsilon_{ij}$
\item the signed graph $G$ on $\{v_0,\ldots,v_\ell\}$ with $E_G^-=\{\{v_i,v_j\}: \epsilon_{ij}<0 \},E_G^+=\{\{v_i,v_j\}:\epsilon_{ij}>0\}$ is \textit{signed-eliminable} in the sense of ~\cite{AbeSignedEliminable}.
\end{enumerate}		
\end{enumerate}
\end{thm}

\begin{remark}
Notice that $\dev(\m_U)=0$ if and only if $m_{ij}-m_{js}+m_{st}-m_{it}=0$ for every four-tuple $(v_i,v_j,v_s,v_t)$ of distinct vertices in $U$.  These equations cut out the linear space $L_U$ parametrized by $m_{ij}=n_i+n_j$ for $\{v_{i},v_{j}\}\subset U$.  Thus $\dev(\m_U)$ can be viewed as a measure of how far $\m_U$ is from the linear space $L_U$; which in turn measures how far $\m_U$ `deviates' from being an ANN multiplicity on the sub-braid arrangement corresponding to $U$.  This is why we call it the \textit{deviation} of $\m$ over $U$.
\end{remark}

The implication (3)$\implies$(1) in Theorem~\ref{thm:1} is the result of Abe-Nuida-Numata~\cite{AbeSignedEliminable}.  The quantity $\dev(\m_U)$ in Theorem~\ref{thm:1}.(2) arises from studying the local and global mixed products (introduced in~\cite{TeraoCharPoly}) of $(A_\ell,\m)$.  The main point of our note is to show that these `deviations' not only detect freeness in the balanced cone but also interact well with the notion of signed-eliminable graphs.

The structure of the paper is as follows.  In Section~\ref{sec:notation} we give some background on arrangements.  The proof of Theorem~\ref{thm:1} is split across sections~\ref{sec:AnyBraidNonFree} and~\ref{sec:SignedEliminable}.  The implication (1)$\implies$ (2) is proved in Theorem~\ref{thm:AnyBraidNonFree}.  We split the proof of (2)$\implies$ (3) into two parts.  The first part, establishing that $\m$ is an ANN multiplicity if the inequalities in (2) are satisfied, is Proposition~\ref{prop:ANNmultiplicity}.  The second part, showing that the inequalities in (2) detect when the associated signed graph is not signed-eliminable, is Proposition~\ref{prop:NotSignedEliminable}.  The final implication (3)$\implies$ (1) is proved in~\cite{AbeSignedEliminable}.  We finish in Section~\ref{sec:FreeVertex} by introducing the notion of a \textit{free vertex} and presenting a conjecture about the structure of all free multiplicities on braid arrangements.

\subsection{Examples}
We provide some computations using Theorem~\ref{thm:1}.  For the braid arrangement $A_\ell$, corresponding to the complete graph $K_{\ell+1}$, we label the vertices of $K_{\ell+1}$ by $v_0,\ldots,v_\ell$ and, given a multiplicity $\m$, we denote by $m_{ij}$ the value of $\m$ on the hyperplane $H_{ij}=V(x_i-x_j)$.  If $U\subset\{v_0,\ldots,v_\ell\}$, we denote by $A_U$ the corresponding sub-braid arrangement of $A_\ell$.

\begin{exm}\label{ex:A33path}
First we consider a family of multiplicities on the $A_3$ arrangement.  Given positive integers $s,t$, define the multiplicity $\m^3_{s,t}$ by $m_{01}=m_{12}=m_{23}=s$ and $m_{02}=m_{03}=m_{13}=t$ (this assigns different multiplicities along two edge-disjoint paths of length three).

The multiplicity $\m^3_{s,t}$ is in the balanced cone of multiplicities if and only if $s\le 2t+1$ and $t\le 2s+1$.  Assuming $\m^3_{s,t}$ is in the balanced cone of multiplicities, we now compute the deviation $\dev(\m^3_{s,t})$.  There are three four cycles: one of these has $\m(C)=|2s-2t|$ while the other two have $\m(C)=|s-t|$.  Hence $\dev(\m^3_{s,t})=6(s-t)^2$.

Now consider the sums $m_{ijk}$ around three cycles.  There are four such sums, two of the form $2s+t$ and two of the form $2t+s$.  So, applying Theorem~\ref{thm:1}, if $\m^3_{s,t}$ is in the balanced cone of multiplicities, it is free if and only if $6(s-t)^2\le 4\cdot 3,$ or $|s-t|\le 1$.  In fact, using the classification from~\cite{A3MultiBraid}, it follows that $\m^3_{s,t}$ is a free multiplicity if and only if $|s-t|\le 1$ (regardless of whether $\m^3_{s,t}$ is in the balanced cone or not).
\end{exm}

\begin{exm}\label{ex:A45cycles}
Next we consider a similar family of multiplicities on the $A_4$ braid arrangement.  Let $C_1$ be the five-cycle traversing the vertices $v_0,v_1,v_2,v_3,v_4,v_0$ in order and $C_2$ be the five-cycle traversing the vertices $v_0,v_2,v_4,v_1,v_3,v_0$ in order; $C_1$ and $C_2$ are edge-disjoint and every edge of $K_5$ is contained in either $C_1$ or $C_2$.  Given positive integers $s,t$ we define the multiplicity $\m^4_{s,t}$ by $\m^4_{s,t}|_{C_1}\equiv s$ and $\m^4_{s,t}|_{C_2}\equiv t$.

Any closed sub-arrangement of $(A_4,\m^4_{s,t})$ of rank three has the form $(A_3,\m^3_{s,t})$ considered in Example~\ref{ex:A33path}.  It follows that $(A_4,\m^4_{s,t})$ is not free if $|s-t|>1$.  So we consider the case when $|s-t|\le 1$.  If $|s-t|\le 1$ then $\m^4_{s,t}$ is in the balanced cone of multiplicities.  We compute $\dev(\m^4_{s,t})$ as follows.  A four-cycle of $K_5$ lies in a unique complete sub-graph on four vertices and each complete sub-graph on four vertices contains three such four-cycles.  As we saw in Example~\ref{ex:A33path}, one of these satisfies $\m(C)=|2s-2t|$ while the other two have $\m(C)=|s-t|$.  So the contribution to $\dev(\m^4_{s,t})$ from each complete sub-graph on four vertices is $6(s-t)^2$.  As there are five such sub-graphs, we have $\dev(\m^4_{s,t})=30(s-t)^2$.

Now we consider the sums $m_{ijk}$ around three cycles.  There are ten such sums, five of the form $2s+t$ and five of the form $2t+s$.  If $|s-t|=1$, then exactly one of $s,t$ is odd so there are precisely five sums around three cycles that are odd.  Hence $\dev(\m^4_{s,t})=30>4\cdot 5$ and $\m^4_{s,t}$ is not free by Theorem~\ref{thm:1}.  So we conclude that $\m^4_{s,t}$ is free if and only if $s=t$.

We also consider why $\m_{s,t}^4$ is not free when $|s-t|=1$ using the criterion of Abe-Nuida-Numata (which is the third statement of Theorem~\ref{thm:1}).  Without loss, suppose $t=s+1$ and let $n_i=\lceil s/2 \rceil$ for $i=0,1,2,3,4$.  If $s$ is even then $m_{ij}=n_i+n_j=s$ for $\{i,j\}\in C_1$ while $m_{ij}=n_i+n_j+1=s+1$ for $\{i,j\}\in C_2$.  In this case the graph $G$ on the vertices $v_0,v_1,v_2,v_3,v_4$ is the (positive) five-cycle given by $C_2$ and is hence not signed-eliminable by the characterization in~\cite{AbeSignedEliminable} (see also Corollary~\ref{cor:SignedEliminableCharacterization}).  Similarly, if $s$ is odd, then $G$ is the negatively signed five-cycle $C_1$.
\end{exm}

\begin{exm}\label{ex:A4NotDetectible}
The following example shows that criterion (2) in Theorem~\ref{thm:1} really does need to be checked on all proper subsets of size at least four.  Consider the $A_4$ arrangement with the multiplicity $\m$ defined by $m_{01}=m_{02}=m_{03}=m_{12}=m_{14}=1$, $m_{04}=m_{13}=m_{23}=m_{24}=2$ and $m_{34}=3$.  We can check that $\m$ lies in the balanced cone of multiplicities.

There are four odd sums around three cycles (so in the notation of Theorem~\ref{thm:1}, $q=4$).  Also, we compute $\dev(\m)=16$.  From Theorem~\ref{thm:1}, we cannot conclude that $(A_4,\m)$ is not free since $q\ell=16$ also in this case.  However, let us consider the $A_3$ sub-arrangement $A_U$ where $U=\{v_0,v_1,v_3,v_4\}$.  Let $\m_U$ be the restricted multiplicity; it also lies in the balanced cone of multiplicities on $A_U$.  All sums around three-cycles are even, and $\dev(\m_U)=8$.  Since $8>0$, it follows from Theorem~\ref{thm:1} that $(A_U,\m_U)$ is not free, hence $(A_4,\m)$ is also not free.
\end{exm}

\section{Notation and preliminaries} \label{sec:notation}
Let $V=\kk^{\ell}$ be a vector space over a field $\kk$ of characteristic zero.  A \textit{central hyperplane arrangement} $\A=\cup_{i=1}^n H_i$ is a union of hyperplanes $H_i\subset V$ passing through the origin in $V$.  In other words, if we let $\{x_1,\ldots,x_{\ell}\}$ be a basis for the dual space $V^*$ and $S=\mbox{Sym}(V^*)\cong \kk[x_1,\ldots,x_l]$, then $H_i=V(\alpha_{H_i})$ for some choice of linear form $\alpha_{H_i}\in V^*$, unique up to scaling.  We will use the language of graphic arrangements for referring to the braid arrangement $A_\ell$ and its sub-arrangements.  Namely, suppose $G=(V_G,E_G)$ is a graph with vertices ordered as $V_G=\{v_0,\ldots,v_\ell\}$, and let $S=\kk[x_0,\ldots,x_\ell]$.  If $\{v_i,v_j\}$ is an edge in $E_G$ then let $H_{ij}=V(x_i-x_j)$.  The graphic arrangement associated to $G$ is $\A_G=\cup_{\{i,j\}\subset E(G)} H_{ij}$.  Clearly $\A_G$ is a sub-arrangement of the full braid arrangement $\A_\ell$, which may be identified with the graphic arrangement corresponding to the complete graph $K_{\ell+1}$ on $(\ell+1)$ vertices.

A \textit{multi-arrangement} is a pair $(\A,\m)$ of a central arrangement $\A=\cup_{i=1}^k H_i$ and a map $\m:\{H_1,\ldots,H_k\}\rightarrow \Z_{\ge 0}$, called a multiplicity.  If $\m\equiv 1$, then $(\A,\m)$ is denoted $\A$ and is called a \textit{simple} arrangement.  If $\A_G$ is a graphic arrangement then the multi-arrangement $(\A_G,\m)$ is equivalent to the information of the edge-labeled graph $(G,\m)$, where $\{v_i,v_j\}$ is labeled by $\m(H_{ij})=m_{ij}$.  We will frequently move back and forth between these notations.  We will always assume that a graph $G$ comes with some ordering $V_G=\{v_0,\ldots,v_\ell\}$ of its vertices and may refer to the vertices simply by their integer labels $\{0,\ldots,\ell\}$.

The \textit{module of derivations} on $S$ is defined by $\mbox{Der}_{\kk}(S)=\bigoplus_{i=1}^{\ell} S\partial_{x_i}$, the free $S$-module with basis $\partial_{x_i}=\partial/\partial x_i$ for $i=1,\ldots,\ell$.  The module $\mbox{Der}_\kk(S)$ acts on $S$ by partial differentiation.  Given a multi-arrangement $(\A,\m)$, our main object of study is the module $D(\A,\m)$ of \textit{logarithmic derivations} of $(\A,\m)$:
\[
D(\A,\m):=\{\theta\in \mbox{Der}_\kk(S): \theta(\alpha_H)\in\langle \alpha_H^{\m(H)}\rangle \text{ for all }H\in\A \},
\]
where $\langle \alpha_H^{\m(H)} \rangle\subset S$ is the ideal generated by $\alpha_H^{\m(H)}$.  If $D(\A,\m)$ is a free $S$-module, then we say $(\A,\m)$ is free or $\m$ is a free multiplicity of the simple arrangement $\A$.  For a simple arrangement, $D(\A,\m)$ is denoted $D(\A)$; if $D(\A)$ is free we say $\A$ is free.

The \textit{intersection lattice} of $\A$ is the ranked poset $L=L(\A)$ consisting of all intersections of hyperplanes of $\A$ ordered with respect to reverse inclusion (the vector space $V$ is included as the `empty' intersection).  We denote by $L_k$ the intersections of rank $k$, where the rank of an intersection is its codimension.  If $X\in L_k$, $\A_X$ denotes the sub-arrangement consisting of hyperplanes which contain $X$, $L_X$ denotes the lattice of $\A_X$, and $\m_X$ denotes the multiplicity function restricted to hyperplanes containing $X$.  If $\A_G$ is a graphic arrangement with lattice $L$ and $H\subset G$ is a connected induced sub-graph of $G$ on $(k+1)$ vertices, then $H$ corresponds to an intersection $X(H)\in L_k$, and the graphic arrangement $\A_H$ is the same as $(\A_G)_{X(H)}$.  In this setting, if $\m$ is a multiplicity on $\A_G$, we denote by $\m_H$ the restriction of $\m$ to the sub-arrangement $\A_H$.

\begin{prop}\cite[Proposition~1.7]{EulerMult}\label{prop:FreeClosed}
If $(\A,\m)$ is a free multi-arrangement, then so is $(\A_X,\m_X)$.
\end{prop}

If $D(\A,\m)$ is free then it has $\ell$ minimal generators as an $S$-module whose degrees are an invariant of $D(\A,\m)$.  These degrees are called the \textit{exponents} of $(\A,\m)$ and we will list them as a non-increasing sequence $(d_1,\ldots,d_\ell)$.  Put $|\m|=\sum_{H\in L_1}\m(H)$.  Then $\sum_{i=1}^\ell d_i=|\m|$ (this follows for instance by an extension of Saito's criterion to multi-arrangements~\cite{ZieglerMulti}).  For a free multi-arrangement, define the $\mbox{k}$th \textit{global mixed product} by
\[
\mbox{GMP}(k)=\sum d_{i_1}d_{i_2}\cdots d_{i_k},
\]
where the sum runs across all $k$-tuples satisfying $1\le i_1<\cdots<i_k\le \ell$.  Furthermore, define the $k$th \textit{local mixed product} by
\[
\mbox{LMP}(k)=\sum_{X\in L_k} d^X_1d^X_2\cdots d^X_k,
\]
where $d^X_1,\ldots,d^k_X$ are the (non-zero) exponents of the rank $k$ sub-arrangement $\A_X$.  We make use of the following result for $k=2$.

\begin{thm}\cite[Corollary~4.6]{TeraoCharPoly}\label{thm:GMP=LMP}
If $(\A,\m)$ is free then $\mbox{GMP}(k)=\mbox{LMP}(k)$ for every $2\le k\le \ell$.
\end{thm}

\section{Deviations and mixed products in the balanced cone}\label{sec:AnyBraidNonFree}
In this section we study the local and global mixed products of multiplicities in the balanced cone.  In particular, we prove the implication (1)$\implies$(2) of Theorem~\ref{thm:1}.  Recall the \textit{balanced cone} of multiplicities on a braid arrangement $A_\ell$ is the set of multiplicities satisfying the three inequalities $m_{ij}+m_{jk}+1\ge m_{ik}, m_{ij}+m_{ik}+1\ge m_{jk}$, and $m_{ik}+m_{jk}+1\ge m_{ij}$ for every triple $0\le i<j<k\le \ell$.  The following proposition (due to Wakamiko) explains why we call this the \textit{balanced} cone; it is because the exponents of every sub-$A_2$ arrangement are as balanced as possible.

\begin{prop}\cite{Wakamiko}\label{prop:A2exponents}
Suppose $H$ is a three-cycle on the vertices $i,j,k$ of the edge-labeled complete graph $(K_{\ell+1},\m)$.  Put $m_{ijk}=m_{ij}+m_{ik}+m_{jk}$.  If $\m$ is in the balanced cone of multiplicities then the (non-zero) exponents of $(\A_H,\m_H)$ are $(\lfloor m_{ijk}/2\rfloor,\lceil m_{ijk}/2 \rceil)$.
\end{prop}

If $\{i,j,k\}$ are vertices of $K_{\ell+1}$ so that $m_{ij}+m_{ik}+m_{jk}$ is odd then we will call $\{i,j,k\}$ an \textit{odd three-cycle}.

\begin{prop}\label{prop:LMPGMP}
Let $(A_\ell,\m)$ be a multi-braid arrangement so that $\m$ is in the balanced cone of multiplicities.  Set $|\m|=\sum_{ij}m_{ij}$ and $m_{ijk}=m_{ij}+m_{jk}+m_{ik}$.  If $q$ is the number of odd three cycles of $\m$, then
\[
\mbox{LMP}(2)=\sum\limits_{0\le i<j<k \le \ell} (m_{ijk}/2)^2+\sum_{\{i,j\}\cap\{s,t\}=\emptyset} m_{ij}m_{st}-q/4
\]
and
\[
\mbox{GMP}(2)\le \dbinom{\ell}{2}\dfrac{|\m|^2}{\ell^2}.
\]
\end{prop}
\begin{proof}
We prove the formula for $\mbox{LMP}(2)$ first.  If $X\in L_2$, then either $(1):X=H_{ij}\cap H_{st}$ for a pair of non-adjacent edges $\{i,j\}$ and $\{s,t\}$ or $(2):X=H_{ij}\cap H_{jk}\cap H_{ik}$ corresponds to a triangle.  In the first case the arrangement is boolean with (non-zero) exponents $(m_{ij},m_{st})$, contributing $m_{ij}m_{st}$ to $\mbox{LMP}(2)$.  In the second case the arrangement is an $A_2$ braid arrangement with exponents $(m_{ijk}/2,m_{ijk}/2)$ if $m_{ijk}$ is even and $((m_{ijk}-1)/2,(m_{ijk}+1)/2)$ if $m_{ijk}$ is odd from Proposition~\ref{prop:A2exponents}.  The former contributes $m_{ijk}^2/4$ to $\mbox{LMP}(2)$ while the latter contributes $m_{ijk}^2/4-1/4$.  This yields the expression for $\mbox{LMP}(2)=\sum_{X\in L_2} d^X_1d^X_2$.

Now consider the inequality for $\mbox{GMP}(2)$.  Supposing $(\A_\ell,\m)$ is free, let $(d_1,\ldots,d_\ell)$ be its (non-zero) exponents, ordered so that $d_1\ge\cdots\ge d_\ell$.  By an extension of Saito's criterion to multi-arrangements~\cite{ZieglerMulti}, $\sum_{i=1}^\ell d_i=|\m|$.  Now, following remarks just after~\cite[Corollary~4.6]{TeraoCharPoly}, we say that $(b_1,\ldots,b_\ell)$ with $b_1\ge\cdots\ge b_\ell$ and $\sum b_i=|\m|$ is `more balanced' than $(d_1,\ldots,d_\ell)$ if $\sum (b_{i+1}-b_i)\le \sum (d_{i+1}-d_i)$.  Then
\[
\sum b_{i_1}\cdots b_{i_k}\ge \sum d_{i_1}\cdots d_{i_k}=\mbox{GMP}(k).
\]
Let $|\m|=k\ell+p$ be the result of dividing $|\m|$ by $\ell$, so $k$ is a positive integer and $0\le p<\ell$.  The `most balanced' distribution of exponents occurs when
\[
b_1=\cdots=b_p=k+1 \qquad \mbox{and} \qquad b_{p+1}=\ldots=b_\ell=k,
\]
so $\sum (b_{i+1}-b_i)=b_{p-1}-b_p$ is zero if $p=0$ and one if $p>0$.  Some algebra yields that, for this choice of exponents, 
\[
\sum_{1\le i< j\le \ell} b_{i}b_{j}=\dbinom{\ell}{2}\dfrac{|\m|^2}{\ell^2}-\dfrac{p(\ell-p)}{2\ell}\le\dbinom{\ell}{2}\dfrac{|\m|^2}{\ell^2},
\]
which gives the result. \qedhere
\end{proof}

\begin{defn}\label{defn:Deviation}
Fix a multiplicity $\m$ on the braid arrangement $A_\ell$ and a subset $U\subset\{v_0,\ldots,v_\ell\}$.  Denote by $C_4(K_{\ell+1})$ the set of four-cycles in $K_{\ell+1}$ and by $C_4(U)$ the set of four-cycles in $K_{\ell+1}$ whose vertices are contained in $U$.  If $C\in C_4(K_{\ell+1})$ traverses the vertices $i,j,s,t$ in order, set  $\m(C)=|m_{ij}-m_{js}+m_{st}-m_{it}|$.  The \textbf{deviation of }$\m$ \textbf{over }$U$ is the sum of squares
\[
\dev(\m_U)=\sum_{C\in C_4(U)} \m(C)^2.
\]
If $U$ consists of all vertices of $K_{\ell+1}$, then we write $\dev(\m)$ instead of $\dev(\m_U)$.
\end{defn}

\begin{thm}\label{thm:AnyBraidNonFree}
Suppose $(A_\ell,\m)$ is a multi-braid arrangement, $\m$ is in the balanced cone of multplicities, and $q$ is the number of odd three cycles of $\m$.  If $\dev(\m)>q\ell$, then $\m$ is not free.
\end{thm}


\begin{remark}
The inequality $\dev(\m)>q\ell$ in Theorem~\ref{thm:AnyBraidNonFree} can be strengthened to $\dev(\m)>q\ell-2p(\ell-p)$, where $p$ is the remainder of $|\m|$ on division by $\ell$.  We will see that the simpler inequality $\dev(\m)>q\ell$ suffices to detect non-freeness.
\end{remark}

\begin{proof}
By Proposition~\ref{prop:LMPGMP}, we know that
\begin{multline*}
\mbox{LMP}(2)-\mbox{GMP}(2)\ge \\ \sum\limits_{0\le i<j<k \le \ell} (m_{ijk}/2)^2+\sum_{\{i,j\}\cap\{s,t\}} m_{ij}m_{st}-\dbinom{\ell}{2}\dfrac{|\m|^2}{\ell^2}-q/4.
\end{multline*}
Our primary claim is
\begin{equation}\label{eq:SOS}
4\ell\left(\sum\limits_{0\le i<j<k \le n} (m_{ijk}/2)^2+\sum_{\{i,j\}\cap\{s,t\}} m_{ij}m_{st}-\dbinom{\ell}{2}\dfrac{|\m|^2}{\ell^2}\right)=\dev(\m).
\end{equation}
Once Equation~\eqref{eq:SOS} is proved, notice that
\[
4\ell(\mbox{LMP}(2)-\mbox{GMP}(2))\ge \dev(\m)-q\ell.
\]
Then Theorem~\ref{thm:GMP=LMP} immediately yields Theorem~\ref{thm:AnyBraidNonFree}.  So we prove Equation~\eqref{eq:SOS}.  We first consider the right hand side, namely the sum $\dev(\m)$.  Since every edge of $K_{\ell+1}$ is contained in $2\binom{\ell-1}{2}$ four-cycles, every pair of disjoint edges is contained in two four-cycles, and every pair of adjacent edges is contained in $(\ell-2)$ four-cycles,
\begin{multline}\label{eq:PMreduce}
\dev(\m)=2\binom{\ell-1}{2}\sum_{0\le i<j\le \ell} m_{ij}^2+4\sum_{\{i,j\}\cap\{s,t\}=\emptyset} m_{ij}m_{st} \\ -2(\ell-2)\sum_{0\le i<j<k\le \ell} (m_{ij}m_{ik}+m_{ij}m_{jk}+m_{ik}m_{jk}).
\end{multline}
Now consider the left hand side of Equation~\eqref{eq:SOS}.  Distributing $4\ell$ yields
\[
\ell\sum\limits_{0\le i<j<k \le \ell} (m_{ijk})^2+4\ell\sum_{\{i,j\}\cap\{s,t\}} m_{ij}m_{st}-2(\ell-1)|\m|^2.
\]
This expression can be re-written in the form of Equation~\eqref{eq:PMreduce} using the following two expressions and simplifying:
\[
\begin{array}{rl}
|\m|^2= & \sum\limits_{0\le i<j\le \ell} m_{ij}^2 +2\sum\limits_{\{i,j\}\cap\{s,t\}=\emptyset} m_{ij}m_{st}\\[10 pt]
&+2\sum\limits_{0\le i<j<k\le \ell} (m_{ij}m_{ik}+m_{ij}m_{jk}+m_{ik}m_{jk})\\[15 pt]
\sum\limits_{0\le i<j<k\le\ell} m_{ijk}^2= & (\ell-1)\sum\limits_{0\le i<j\le \ell} m_{ij}^2\\[10 pt]
&+2\sum\limits_{0\le i<j<k\le \ell}(m_{ij}m_{ik}+m_{ij}m_{jk}+m_{ik}m_{jk}).\qedhere
\end{array}
\]
\end{proof}

As an immediate corollary we obtain (1)$\implies$(2) in Theorem~\ref{thm:1}.

\begin{cor}\label{cor:1to2}
If $(A_\ell,\m)$ is free then $\dev(\m_U)\le q_U(|U|-1)$ for every subset $U\subset\{v_0,\ldots,v_\ell\}$.
\end{cor}
\begin{proof}
By Proposition~\ref{prop:FreeClosed}, freeness of $(A_\ell,\m)$ implies freeness of $(A_U,\m_U)$ for every subset $U\subset\{v_0,\ldots,v_\ell\}$.  By Theorem~\ref{thm:AnyBraidNonFree}, we must have $\dev(\m_U)\le q_U(|U|-1)$ for every subset $U\subset\{v_0,\ldots,v_\ell\}$ as well.
\end{proof}

\section{From deviations to ANN multiplicities}\label{sec:ANNMultiplicities}

In this section we prove the first part of the implication (2)$\implies$(3) in Theorem~\ref{thm:1}.  Namely, we prove that for a multiplicity $\m$ in the central cone, the inequalities $\dev(\m)\le q_U(|U|-1)$ on deviations are enough to guarantee that $\m$ is an ANN multiplicity.  In fact, we show that it is enough to have these inequalities on subsets of size four.

Recall from the introduction that we call $\m$ an ANN multiplicity on $A_\ell$ if $\m$ is in the balanced cone of multiplicities and there exist non-negative integers $n_0,\ldots,n_\ell$ and $\epsilon_{ij}\in\{-1,0,1\}$ so that $m_{ij}=n_i+n_j+\epsilon_{ij}$ for $0\le i<j\le \ell$.

\begin{lem}\label{lem:K4restrictions}
Suppose $\m$ is a multiplicity on $A_3$ and $\dev(\m)\le 3q$.  Then $\m(C)\le 2$ for each four-cycle $C$ in $K_4$.
\end{lem}
\begin{proof}
There are three four-cycles.  Set
\[
\begin{array}{rl}
T_1= & m_{01}-m_{12}+m_{23}-m_{03}\\
T_2= & m_{13}-m_{01}+m_{02}-m_{23}\\
T_3= & m_{13}-m_{12}+m_{02}-m_{03}.
\end{array}
\]
Notice $T_1+T_2=T_3$, and $\dev(\m)=T_1^2+T_2^2+T_3^2$.  Now, suppose without loss that $|T_3|\ge 3$.  Then either $|T_1|\ge 2$ or $|T_2|\ge 2$.  But then $P(\m)\ge 13$, contradicting that $\dev(\m)\le 3q\le 12$ (since $q\le 4$).
\end{proof}

\begin{prop}\label{prop:ANNmultiplicity}
Let $(A_\ell,\m)$ be a multi-braid arrangement so that $\m$ is in the balanced cone of multiplicities and $\dev(\m_U)\le 3q_U$ for every subset $U\subset\{v_0,\ldots,v_\ell\}$ with $|U|=4$.  Then $\m$ is an ANN multiplicity.
\end{prop}

\begin{proof}
We need only show that there exist non-negative integers $n_i$ for $i=0,\ldots,\ell$ and integers $\epsilon_{ij}\in\{-1,0,1\}$ for $0\le i<j\le \ell$ so that $m_{ij}=n_i+n_j+\epsilon_{ij}$.  By Lemma~\ref{lem:K4restrictions}, we must have $\m(C)\le 2$ for every four-cycle $C\in C_4(K_{\ell+1})$.  We use this condition to provide an inductive algorithm producing the integers $n_0,\ldots,n_\ell$.

If $\ell=2$, set $n_0= \left\lceil\dfrac{m_{01}+m_{02}-m_{12}}{2}\right\rceil, n_1=\left\lceil\dfrac{m_{01}+m_{12}-m_{02}}{2}\right\rceil,$ and $ n_2= \left\lceil\dfrac{m_{02}+m_{12}-m_{01}}{2}\right\rceil.$  Since $\m$ is in the balanced cone, $n_i\ge 0$ for $i=0,1,2$.  Moreover, $m_{ij}=n_i+n_j+\epsilon_{ij}$, where $\epsilon_{ij}\in\{-1,0\}$.  

Now assume $\ell>2$.  We make an initial guess at what the non-negative integers $n_0,\ldots,n_\ell$ and $\epsilon_{ij}$ should be, and then adjust as necessary.  By induction on $\ell$, there exist non-negative integers $\n_0,\ldots,\n_{\ell-1}$ and $\nepsilon_{ij}\in\{-1,0,1\}$ such that $m_{ij}=\n_i+\n_j+\nepsilon_{ij}$ for $0\le i<j\le\ell-1$.  Let $\n_\ell$ be a non-negative integer satisfying $\n_\ell+\n_i\ge m_{i\ell}-1$ and set $\nepsilon_{i\ell}=m_{i\ell}-(\n_i+\n_\ell)$ for every $i<\ell$, so $m_{i\ell}=\n_i+\n_\ell+\nepsilon_{i\ell}$.  By the choice of $\n_\ell$, we have $\nepsilon_{i\ell}\le 1$ for all $i<\ell$.

Now suppose there is an index $0\le j<\ell$ so that $\nepsilon_{j\ell}\le -2$.  Our goal is to decrease either $\n_\ell$ or $\n_j$ by one, thereby increasing $\nepsilon_{j\ell}$, without disturbing any of the hypotheses made so far, namely
\begin{align*}
&\n_i\ge 0\mbox{ for all } 0\le i\le \ell,\nonumber  \\
&\nepsilon_{i\ell}\le 1\mbox{ for all } i<\ell,\tag{$\star$}\label{assumptions} \\
&\nepsilon_{st}\in\{-1,0,1\}\mbox{ for all } 0\le s<t\le \ell-1. 
\end{align*}
First we assume $\n_\ell>0$ and try to decrease $\n_\ell$ by one.  We can do this without disturbing assumptions~\eqref{assumptions} provided there is no index $s$ so that $\epsilon_{s\ell}=1$.  So, assume that there is an index $0\le s<\ell$ so that $\epsilon_{s\ell}=1$.  We claim that in this situation, $\epsilon_{st}\ge 0$ for every $t\neq s$.  Suppose to the contrary that there is an index $t$ so that $\epsilon_{st}=-1$ and consider the four-cycle $C:\ell\to s\to t\to j\to \ell$.  Then
\[
\begin{array}{rl}
\m(C) & = |\nepsilon_{s\ell}-\nepsilon_{j\ell}+\nepsilon_{jt}-\nepsilon_{st}|\\
& \ge 1+2+\nepsilon_{jt}+1\\
& \ge 3,
\end{array}
\]
since $\nepsilon_{jt}\in\{-1,0,1\}$ by the inductive hypothesis.  This contradicts our assumption that $\m(C)\le 2$.  So it follows that $\nepsilon_{st}\in\{0,1\}$ for all $t$.  Thus we may increase $\n_s$ by one, thereby decreasing $\nepsilon_{st}$ by one for every $t\neq s$, without disturbing the hypothesis that $\nepsilon_{st}\in\{-1,0,1\}$.  Since we can apply this argument at every index $s$ so that $\nepsilon_{s\ell}=1$, we may assume $\nepsilon_{s\ell}\le 0$ for every $0\le s<\ell$.  Hence, if $\n_\ell>0$, it is now clear that we can decrease $\n_\ell$ by one without disturbing assumptions~\eqref{assumptions}.

Now assume that $\n_\ell=0$.  Then, for any $s<\ell$,
\[
\begin{array}{rl}
m_{s\ell}+m_{j\ell}-m_{js} & =(\n_s+\nepsilon_{s\ell})+(\n_j+\nepsilon_{j\ell})-(\n_j+\n_s+\nepsilon_{js})\\
&=\nepsilon_{s\ell}+\nepsilon_{j\ell}-\nepsilon_{js}\\
&\le 0-2-\nepsilon_{js}\\
&\le -1,
\end{array}
\]
since $\nepsilon_{js}\in\{-1,0,1\}$ by the inductive hypothesis.  Since $\m$ is in the balanced cone, we must have equality for all of these, so $\epsilon_{js}=-1$ for every $s\neq j$, $s<\ell$.  If $\n_j=0$ as well, then $m_{j\ell}=\n_j+\n_\ell+\epsilon_{j\ell}\le -2$, contradicting that $m_{j\ell}$ is non-negative.  Hence $\n_j>0$ and we can decrease $\n_j$ by one without disturbing any of assumptions~\eqref{assumptions}.

In either case, we have shown how to increase $\nepsilon_{j\ell}$ if $\nepsilon_{j\ell}\le -2$ without disturbing assumptions~\eqref{assumptions}.  So we iterate the above arguments until $\nepsilon_{j\ell}\ge -1$ for every $j<\ell$, then set $n_i=\n_i$ for $0\le i\le \ell$ and $\nepsilon_{ij}=\epsilon_{ij}$ for $0\le i<j\le\ell$.  This completes the algorithm and the proof.
\end{proof}

With Proposition~\ref{prop:ANNmultiplicity}, we now prove (1)$\iff$(3) in Theorem~\ref{thm:1}.  Most of the heavy lifting is done by Abe-Nuida-Numata in~\cite{AbeSignedEliminable}.

\begin{cor}\label{cor:1iff3}
Suppose $\m$ is in the balanced cone of multiplicities on the $A_\ell$ braid arrangement.  Then $(A_\ell,\m)$ is free if and only if $\m$ is an ANN multiplicity and the signed graph with $E_G^+=\{\{v_i,v_j\}:\epsilon_{ij}=-1\}$ and $E^-_G=\{\{v_i,v_j\}:\epsilon_{ij}=1\}$ is signed-eliminable.
\end{cor}
\begin{proof}
If $(A_\ell,\m)$ is free and $\m$ is in the balanced cone, then $\dev(\m_U)\le 3q_U$ for every subset $U\subset\{v_0,\ldots,v_\ell\}$ of size four by Corollary~\ref{cor:1to2}.  By Proposition~\ref{prop:ANNmultiplicity}, $\m$ is an ANN multiplicity.  By~\cite[Theorem~0.3]{AbeSignedEliminable}, the signed graph with $E_G^+=\{\{v_i,v_j\}:\epsilon_{ij}=-1\}$ and $E^-_G=\{\{v_i,v_j:\epsilon_{ij}=1\}$ is signed-eliminable.  For the converse, if $\m$ is an ANN multiplicity associated to a signed-eliminable graph, then $(A_\ell,\m)$ is free by~\cite[Theorem~0.3]{AbeSignedEliminable}.
\end{proof}

\begin{remark}
In the result~\cite[Theorem~0.3]{AbeSignedEliminable}, Abe-Nuida-Numata do not have the condition that $\m$ is in the balanced cone.  However, this turns out to be a necessary condition for their arguments~\cite{AbeCommunication}.  Furthermore their arguments, using addition-deletion techniques for multi-arrangements from~\cite{EulerMult}, work for any ANN multiplicity as we have defined it~\cite{AbeCommunication}.
\end{remark}

\section{Detecting signed-eliminable graphs}\label{sec:SignedEliminable}
In this section we finish the proof of Theorem~\ref{thm:1}.  We have already shown in Proposition~\ref{prop:ANNmultiplicity} that if the inequalities of Theorem~\ref{thm:1}.(2) are satisfied then $\m$ is an ANN multiplicity.  Now we show that these inequalities also detect when the associated signed graph is not signed-eliminable.  We follow the presentation of signed-eliminable graphs from~\cite{NuidaSignedEliminable,AbeSignedEliminable}.

Let $G$ be a \textit{signed} graph on $\ell+1$ vertices.  That is, each edge of $G$ is assigned either a $+$ or a $-$, and so the edge set $E_G$ decomposes as a disjoint union $E_G=E_G^+\cup E_G^-$.  Define
\[
m_G(ij)=\left\lbrace
\begin{array}{rl}
1 & \{i,j\}\in E^+_G\\
-1 & \{i,j\}\in E^-_G\\
0 & \text{otherwise}.
\end{array}
\right.
\]
The graph $G$ is \textit{signed-eliminable} with \textit{signed-elimination ordering} $\nu:V(G)\rightarrow \{0,\ldots,\ell\}$ if $\nu$ is bijective and, for every three vertices $v_i,v_j,v_k\in V(G)$ with $\nu(v_i),\nu(v_j)<\nu(v_k)$, the induced sub-graph $G|_{v_i,v_j,v_k}$ satisfies the following conditions.
\begin{itemize}
	\item For $\sigma\in\{+,-\}$, if $\{v_i,v_k\}$ and $\{v_j,v_k\}$ are edges in $E^\sigma_G$ then $\{v_i,v_j\}\in E^\sigma_G$
	\item For $\sigma\in\{+,-\}$, if $\{v_k,v_i\}\in E^\sigma_G$ and $\{v_i,v_j\}\in E^{-\sigma}_G$ then $\{v_k,v_j\}\in E_G$
\end{itemize}
These two conditions generalize the notion of a graph possessing an elimination ordering, which is equivalent to the graph being chordal.  A graph is chordal if and only if it has no induced sub-graph which is a cycle of length at least four.  In~\cite{NuidaSignedEliminable}, Nuida establishes a similar characterization for signed-eliminable graphs, to which we now turn.

\begin{defn}\label{defn:ObstructionGraphs}
\begin{enumerate}
\item A graph with $(\ell+1)$ vertices $v_0,v_1,\ldots,v_{\ell}$ with $\ell\ge 3$ is a $\sigma$-mountain, where $\sigma\in\{+,-\}$, if $\{v_0,v_i\}\in E^\sigma_G$ for $i=2,\ldots,\ell-1$, $\{v_i,v_{i+1}\}\in E^{-\sigma}_G$ for $i=1,\ldots,\ell-1$, and no other pair of vertices is joined by an edge.  (See Figure~\ref{fig:SMSH} - edges of sign $\sigma$ are denoted by a single edge and edges of sign $-\sigma$ are denoted by a doubled edge.)
\item A graph with $(\ell+1)$ vertices $v_0,v_1,v_2,\ldots,v_{\ell}$ with $\ell\ge 3$ is a $\sigma$-hill, where $\sigma\in\{+,-\}$, if $\{v_0,v_1\}\in E^\sigma_G,\{v_0,v_i\}\in E^{\sigma}_G$ for $i=2,\ldots,\ell-1$, $\{v_1,v_i\}\in E^\sigma_G$ for $i=3,\ldots,\ell$, $\{v_i,v_{i+1}\}\in E^{-\sigma}_G$ for $i=2,\ldots,\ell-1$, and no other pair of vertices is connected by an edge.  (See Figure~\ref{fig:SMSH}.)
\item A graph with $(\ell+1)$ vertices $v_0,\ldots,v_{\ell}$ with $\ell\ge 2$ is a $\sigma$-cycle if $\{v_i,v_{i+1}\}\in E^\sigma_G$ for $i=0,\ldots,\ell-1$, $\{v_0,v_{\ell}\}\in E^\sigma_G$, and no other pair of vertices is connected by an edge.
\end{enumerate}
\end{defn}

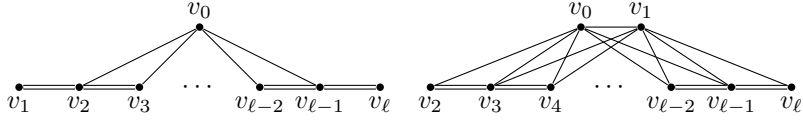
\begin{figure}
\begin{tikzpicture}[scale=.8]
\tikzstyle{dot}=[circle,fill=black,inner sep=1 pt];

\node[dot] (w) at (0,1) {};
\node[dot] (1) at (-3,0) {};
\node[dot] (2) at (-2,0) {};
\node[dot] (3) at (-1,0) {};
\node[dot] (n-2) at (1,0) {};
\node[dot] (n-1) at (2,0) {};
\node[dot] (n) at (3,0) {};
\node at (0,0) {$\cdots$};

\draw (1)node[below]{$v_1$} (2)node[below]{$v_2$} (3)node[below]{$v_3$} (n-2)node[below]{$v_{\ell-2}$} (n-1)node[below]{$v_{\ell-1}$} (n)node[below]{$v_\ell$};
\draw (w)node[above]{$v_0$}--(2) (w)--(3) (w)--(n-2) (w)--(n-1) ;
\draw[double distance=1 pt] (1)--(2)--(3) (n-2)--(n-1)--(n);
\end{tikzpicture}
\begin{tikzpicture}[scale=.8]
\tikzstyle{dot}=[circle,fill=black,inner sep=1 pt];

\node[dot] (w1) at (-.5,1) {};
\node[dot] (w2) at (.5,1) {};
\node[dot] (1) at (-3,0) {};
\node[dot] (2) at (-2,0) {};
\node[dot] (3) at (-1,0) {};
\node[dot] (n-2) at (1,0) {};
\node[dot] (n-1) at (2,0) {};
\node[dot] (n) at (3,0) {};
\node at (0,0) {$\cdots$};

\draw (1)node[below]{$v_2$} (2)node[below]{$v_3$} (3)node[below]{$v_4$} (n-2)node[below]{$v_{\ell-2}$} (n-1)node[below]{$v_{\ell-1}$} (n)node[below]{$v_{\ell}$};
\draw (w1)node[above]{$v_0$}--(1) (w1)--(2) (w1)--(3) (w1)--(n-2) (w1)--(n-1);
\draw (w2)node[above]{$v_1$}--(2) (w2)--(3) (w2)--(n-2) (w2)--(n-1) (w2)--(n);
\draw (w1)--(w2);
\draw[double distance=1 pt] (1)--(2)--(3) (n-2)--(n-1)--(n);
\end{tikzpicture}

\caption{$\sigma$-mountain (at left) and $\sigma$-hill (at right)}\label{fig:SMSH}

\end{figure}

\begin{thm}\cite[Theorem~5.1]{NuidaSignedEliminable}\label{thm:SignedEliminableCharacterization}
	Let $G$ be a signed graph.  Then $G$ is signed-eliminable if and only if the following three conditions are satisfied.
	\begin{enumerate}
		\item[$(C1)$] Both $G_+$ and $G_-$ are chordal.
		\item[$(C2)$] Every induced sub-graph on four vertices is signed eliminable.
		\item[$(C3)$] No induced sub-graph of $G$ is a $\sigma$-mountain or a $\sigma$-hill.
	\end{enumerate}
\end{thm}

All signed-eliminable graphs on four vertices are listed (with an elimination ordering) in~\cite[Example~2.1]{AbeSignedEliminable}, along with those which are not signed-eliminable.  For use in the proof of Corollary~\ref{cor:SignedEliminableCharacterization}, we also list those graphs which are not signed-eliminable in Table~\ref{tbl:NonSignedEliminable}.  The property of being signed-eliminable is preserved under interchanging $+$ and $-$. Consequently, we list these graphs in Table~\ref{tbl:NonSignedEliminable} up to automorphism with the convention that a single edge takes one of the signs $+,-$, while a double edge takes the other sign.

\begin{table}
	\begin{tabular}{cccccc}
		
		\begin{tikzpicture}[scale=1.0]
		\tikzstyle{dot}=[circle,fill=black,inner sep=1 pt];
		
		\node[dot] (1) {};
		\node[dot] (2)[left of=1]{};
		\node[dot] (3)[below of=2]{};
		\node[dot] (4)[below of=1]{};
		
		\draw (1)--(4) (3)--(2);
		\draw[double distance=1 pt] (3)--(4);
		\end{tikzpicture}
		&
		\begin{tikzpicture}[scale=1.0]
		\tikzstyle{dot}=[circle,fill=black,inner sep=1 pt];
		
		\node[dot] (1) {};
		\node[dot] (2)[left of=1]{};
		\node[dot] (3)[below of=2]{};
		\node[dot] (4)[below of=1]{};
		
		\draw (1)--(2) (1)--(3);
		\draw[double distance=1 pt] (1)--(4);
		\end{tikzpicture}
		&
		\begin{tikzpicture}[scale=1.0]
		\tikzstyle{dot}=[circle,fill=black,inner sep=1 pt];
		
		\node[dot] (1) {};
		\node[dot] (2)[left of=1]{};
		\node[dot] (3)[below of=2]{};
		\node[dot] (4)[below of=1]{};
		
		\draw (1)--(2)--(3)--(4)--(1);
		\end{tikzpicture}
		&
		\begin{tikzpicture}[scale=1.0]
		\tikzstyle{dot}=[circle,fill=black,inner sep=1 pt];
		
		\node[dot] (1) {};
		\node[dot] (2)[left of=1]{};
		\node[dot] (3)[below of=2]{};
		\node[dot] (4)[below of=1]{};
		
		\draw (4)--(1)--(2)--(3);
		\draw[double distance=1 pt] (3)--(4);
		\end{tikzpicture}
		&
		\begin{tikzpicture}[scale=1.0]
		\tikzstyle{dot}=[circle,fill=black,inner sep=1 pt];
		
		\node[dot] (1) {};
		\node[dot] (2)[left of=1]{};
		\node[dot] (3)[below of=2]{};
		\node[dot] (4)[below of=1]{};
		
		\draw (1)--(4)--(2)--(3);
		\draw[double distance=1 pt] (3)--(4);
		\end{tikzpicture}
		&	
		\begin{tikzpicture}[scale=1.0]
		\tikzstyle{dot}=[circle,fill=black,inner sep=1 pt];
		
		\node[dot] (1) {};
		\node[dot] (2)[left of=1]{};
		\node[dot] (3)[below of=2]{};
		\node[dot] (4)[below of=1]{};
		
		\draw (2)--(4)--(3);
		\draw[double distance=1 pt] (2)--(3) (1)--(4);
		\end{tikzpicture}
		\\
		\begin{tikzpicture}[scale=1.0]
		\tikzstyle{dot}=[circle,fill=black,inner sep=1 pt];
		
		\node[dot] (1) {};
		\node[dot] (2)[left of=1]{};
		\node[dot] (3)[below of=2]{};
		\node[dot] (4)[below of=1]{};
		
		\draw (1)--(4) (3)--(2);
		\draw[double distance=1 pt] (1)--(2) (3)--(4);
		\end{tikzpicture}
		&
		\begin{tikzpicture}[scale=1.0]
		\tikzstyle{dot}=[circle,fill=black,inner sep=1 pt];
		
		\node[dot] (1) {};
		\node[dot] (2)[left of=1]{};
		\node[dot] (3)[below of=2]{};
		\node[dot] (4)[below of=1]{};
		
		\draw (1)--(2)--(3)--(4)--(1);
		\draw[double distance=1 pt] (1)--(3);
		\end{tikzpicture}
		&
		\begin{tikzpicture}[scale=1.0]
		\tikzstyle{dot}=[circle,fill=black,inner sep=1 pt];
		
		\node[dot] (1) {};
		\node[dot] (2)[left of=1]{};
		\node[dot] (3)[below of=2]{};
		\node[dot] (4)[below of=1]{};
		
		\draw (1)--(2)--(3)--(4);
		\draw[double distance=1 pt] (3)--(1)--(4);
		\end{tikzpicture}
		&
		\begin{tikzpicture}[scale=1.0]
		\tikzstyle{dot}=[circle,fill=black,inner sep=1 pt];
		
		\node[dot] (1) {};
		\node[dot] (2)[left of=1]{};
		\node[dot] (3)[below of=2]{};
		\node[dot] (4)[below of=1]{};
		
		\draw (4)--(1)--(3)--(2);
		\draw[double distance=1 pt] (1)--(2) (3)--(4);
		\end{tikzpicture}
		&
		\begin{tikzpicture}[scale=1.0]
		\tikzstyle{dot}=[circle,fill=black,inner sep=1 pt];
		
		\node[dot] (1) {};
		\node[dot] (2)[left of=1]{};
		\node[dot] (3)[below of=2]{};
		\node[dot] (4)[below of=1]{};
		
		\draw (1)--(4)--(2)--(3);
		\draw[double distance=1 pt] (2)--(1)--(3)--(4);
		\end{tikzpicture}
		&
		\begin{tikzpicture}[scale=1.0]
		\tikzstyle{dot}=[circle,fill=black,inner sep=1 pt];
		
		\node[dot] (1) {};
		\node[dot] (2)[left of=1]{};
		\node[dot] (3)[below of=2]{};
		\node[dot] (4)[below of=1]{};
		
		\draw (1)--(2)--(4)--(3)--(1);
		\draw[double distance=1 pt] (2)--(3) (1)--(4);
		\end{tikzpicture}
	\end{tabular}
	
	\caption{Graphs on four vertices which are \textit{not} signed-eliminable}\label{tbl:NonSignedEliminable}
\end{table}

\begin{cor}\label{cor:SignedEliminableCharacterization}
	Let $G$ be a signed graph.  Then $G$ is signed-eliminable if and only if the following three conditions are satisfied.
	\begin{enumerate}
		\item[$(C1')$] No induced sub-graph of $G$ is a $\sigma$-cycle of length $>3$.
		\item[$(C2)$] Every induced sub-graph on four vertices is signed eliminable.
		\item[$(C3)$] No induced sub-graph of $G$ is a $\sigma$-mountain or a $\sigma$-hill.
	\end{enumerate}
\end{cor}

\begin{proof}
Clearly $(C1)$ from Theorem~\ref{thm:SignedEliminableCharacterization} implies $(C1')$.  We show that $(C1')$ and $(C2)$ imply condition $(C1)$.  Assume for contradiction that $E^{\sigma}_G$, $\sigma\in\{-,+\}$, is a cycle of length $\ell+1>3$ and $V_G=\{v_0,\ldots,v_{\ell}\}$ where $\{v_i,v_{i+1}\}\in E^{\sigma}_G$ for $i=0,\ldots,\ell-1$ and $\{v_0,v_{\ell}\}\in E^{\sigma}_G$.  If $E^{-\sigma}_G=\emptyset$ then $G$ is a $\sigma$-cycle which is forbidden by $(C1')$, so we assume $E^{-\sigma}_G\neq\emptyset$.  Let $m$ be the maximal integer so that there is a sequence of consecutive vertices $v_{i},v_{i+1}\ldots,v_{i+m-1}$ so that the induced sub-graph on these consecutive vertices consists only of edges in $E^{\sigma}_G$.  Since $E^{-\sigma}_G\neq\emptyset$, $m<\ell+1$.  Relabel the vertices so that $v_0,\ldots,v_{m-1}$ are the vertices of a maximal induced sub-graph with edges only in $E^{\sigma}_G$.  If $m=2$ or $m=3$, then the induced sub-graph on $v_0,v_1,v_2,v_3$ consists of the three $\sigma$ edges $\{v_0,v_1\},\{v_1,v_2\},\{v_2,v_3\}$ along with at least one $-\sigma$ edge.  No such graph is signed eliminable (see Table~\ref{tbl:NonSignedEliminable}).  So $m\ge 4$.  Now consider the induced sub-graph $H$ on $v_0,v_1,\ldots,v_{m-1},v_{m}$.  By definition of $m$, $H$ has exactly one $-\sigma$ edge, namely $\{v_0,v_{m}\}$.  But then the induced sub-graph on $v_0,v_1,v_{m-1},v_m$ consists of the two $\sigma$ edges $\{v_0,v_1\},\{v_{m-1},v_m\}$ and the $-\sigma$ edge $\{v_0,v_m\}$, which is not signed-eliminable.  It follows that $E^\sigma_G$ cannot have a cycle of length $>3$, so $E^\sigma_G$ is chordal.
\end{proof}

The following proposition proves the implication (2)$\implies$(3) in Theorem~\ref{thm:1}, thereby completing the proof of Theorem~\ref{thm:1}.

\begin{prop}\label{prop:NotSignedEliminable}
Suppose $n_0,\ldots,n_{\ell}$ are non-negative integers, $G$ is a signed graph on $v_0,\ldots,v_\ell$, and let $\m$ be the multiplicity on $A_\ell$ given by $m_{ij}=n_i+n_j+m_G(ij)$.  If $G$ is not signed-eliminable, then there is a subset $U\subset\{0,\ldots,\ell\}$ so that $\dev(\m_U)>q_U\cdot (|U|-1)$.
\end{prop}
\begin{proof}
	Notice that, for a four-cycle traversing $i,j,s,t$ in order,
	\[
	\m(C)=|m_{ij}-m_{js}+m_{st}-m_{it}|=|m_G(ij)-m_G(js)+m_G(st)-m_G(it)|.
	\]
	Furthermore, for a three-cycle $\{i,j,k\}$,
	\[
	m_{ij}+m_{ik}+m_{jk}=2(n_i+n_j+n_k)+m_G(ij)+m_G(ik)+m_G(jk).
	\]
	It follows that the values of $\dev(\m)=\sum \m(C)^2$ and $q\ell=(\# \mbox{ odd three cycles})\cdot\ell$ from Theorem~\ref{thm:AnyBraidNonFree} may be determined after replacing $m_{ij}$ by $m_G(ij)$, which takes values only in $\{-1,0,1\}$.  Hereafter we write $\dev(G)$ for $\dev(\m)$ and $q_G$ for $q$ to emphasize their dependence only on the signed graph $G$.  If $U\subset\{v_0,\ldots,v_\ell\}$, we let $\dev(G_U)$ represent $\dev(\m_U)$ to emphasize dependence only on $G$ and the subset $U$.  As usualy, $q_U$ denotes the number of odd three cycles contained in $U$.
	
	Now, if $G$ is not signed eliminable then by Corollary~\ref{cor:SignedEliminableCharacterization} $G$ contains an induced sub-graph $H$ which is
	\begin{itemize}
		\item a signed graph on four vertices which is not signed-eliminable,
		\item a $\sigma$-cycle of length $>3$,
		\item a $\sigma$-hill,
		\item or a $\sigma$-mountain.
	\end{itemize}
	We assume $G=H$ and show that $DV(G)>q_G\ell$ in each of these cases, where $\ell$ is one less than the number of vertices of $G$.  The inequality $DV(G)>3q_G$ can easily be verified by hand for each of the twelve graphs on four vertices which are not signed-eliminable (see Table~\ref{tbl:NonSignedEliminable}); this is also done in~\cite[Corollary~6.2]{A3MultiBraid}.  If $G$ is a $\sigma$-cycle, $\sigma$-mountain, or $\sigma$-hill on $(\ell+1)$ vertices we will show that $\dev(G)$ and $q_G$ are given by the formulas:
	\begin{align}
	\label{eq:P1}
	\dev(G)= & \ell^3-2\ell^2-\ell+2\\
	\label{eq:q1}
	q_G= & \ell^2-2\ell-3.
	\end{align}
	Given these formulas, note that $\dev(G)=q\ell+2(\ell+1)>q\ell$, thus proving the result.  We prove Equations~\eqref{eq:P1} and~\eqref{eq:q1} for the $\sigma$-cycle directly, relying on the two additional formulas:
	\begin{align}
	\label{eq:P}
	\dev(G)=\sum_{U\subset V_G,|U|=4} \dev(G_U)\\ 
	\label{eq:q}
	q_G=(\sum_{U\subset V_G,|U|=4} q_U)/(\ell-2).
	\end{align}
	Equation~\eqref{eq:P} follows since each four-cycle is contained in a unique induced sub-graph on four vertices and Equation~\eqref{eq:q} follows since each three-cycle appears in $(\ell-2)$ sub-graphs on four vertices.  Using these equations, it suffices to identify all possible types of induced sub-graphs of the $\sigma$-cycle on four vertices, how many of each type there are, and compute $\dev(G_U)$ and $q_U$ for each of these.  Then we use Equation~\eqref{eq:P} to compute $\dev(G)$ and Equation~\eqref{eq:q} to compute $q_G$.  The list of all possible induced sub-graphs with four vertices of a $\sigma$-cycle on $(\ell+1)$ vertices are listed in Table~\ref{tbl:Cycle}.  The number of sub-graphs of each type is listed in the second column, while the third and fourth columns record $q_U$ and $\dev(G_U)$, respectively, for each type of sub-graph.  The final row records the total number of sub-graphs on four vertices, the number of odd three-cycles, and the deviation of $\m$, $\dev(\m)=\sum_{C\in C_4(K_{\ell+1})} \m(C)^2$.  We find that $\dev(\m)=\ell^3-2\ell^2-\ell+2$ and $q=\ell^2-2\ell-3$, proving Equations~\eqref{eq:P1} and~\eqref{eq:q1} for the $\sigma$-cycle.  The same computations can be done to prove Equations~\eqref{eq:P1} and~\eqref{eq:q1} for the $\sigma$-hill and $\sigma$-mountain; for the convenience of the reader we collect these in Appendix~\ref{app:1}. \qedhere	
\end{proof}

\begin{table}
	\renewcommand{\arraystretch}{3}
	\begin{tabular}{c|c|c|c}
		\multicolumn{4}{c}{
			$\sigma$-cycle of length $(\ell+1)$
		}
		\\
		\hline
		Type of sub-graph & Count & $q_U$ & $\dev(G_U)$ \\
		
		\hline
		
		\raisebox{-.5\height}{
			\begin{tikzpicture}[scale=1.0]
			\tikzstyle{dot}=[circle,fill=black,inner sep=1 pt];
			
			\node[dot] (1) {};
			\node[dot] (2)[left of=1]{};
			\node[dot] (3)[below of=2]{};
			\node[dot] (4)[below of=1]{};

			\end{tikzpicture}
		} 
		& 
		$\dbinom{\ell-4}{2}+\dbinom{\ell-3}{2}$ & $0$ & $0$ \\[10 pt]
		
		\hline
		
		\raisebox{-.5\height}{
			\begin{tikzpicture}[scale=1.0]
			\tikzstyle{dot}=[circle,fill=black,inner sep=1 pt];
			
			\node[dot] (1) {};
			\node[dot] (2)[left of=1]{};
			\node[dot] (3)[below of=2]{};
			\node[dot] (4)[below of=1]{};
			
			\draw (1)--(2);
			\end{tikzpicture}
		}
		&
		$(\ell+1)\dbinom{\ell-4}{2}$ & $2$ & $2$ \\[10 pt]
		
		\hline
		
		\raisebox{-.5\height}{
			\begin{tikzpicture}[scale=1.0]
			\tikzstyle{dot}=[circle,fill=black,inner sep=1 pt];
			
			\node[dot] (1) {};
			\node[dot] (2)[left of=1]{};
			\node[dot] (3)[below of=2]{};
			\node[dot] (4)[below of=1]{};
			
			\draw (1)--(2) (3)--(4);
			\end{tikzpicture}
		}
		&
		$\dfrac{(\ell+1)(\ell-4)}{2}$ & $4$ & $8$ \\[10 pt]
		
		\hline
		
		\raisebox{-.5\height}{
			\begin{tikzpicture}[scale=1.0]
			\tikzstyle{dot}=[circle,fill=black,inner sep=1 pt];
			
			\node[dot] (1) {};
			\node[dot] (2)[left of=1]{};
			\node[dot] (3)[below of=2]{};
			\node[dot] (4)[below of=1]{};
			
			\draw (1)--(2)--(3);
			\end{tikzpicture}
		}
		
		&
		$(\ell+1)(\ell-4)$ & $2$ & $2$ \\[10 pt]
		
		\hline
		
		\raisebox{-.5\height}{
			\begin{tikzpicture}[scale=1.0]
			\tikzstyle{dot}=[circle,fill=black,inner sep=1 pt];
			
			\node[dot] (1) {};
			\node[dot] (2)[left of=1]{};
			\node[dot] (3)[below of=2]{};
			\node[dot] (4)[below of=1]{};
			
			\draw (1)--(2)--(3)--(4);
			\end{tikzpicture}
		}
		&
		$\ell+1$ & $2$ & $6$\\[10 pt]
		
		\hline
		
		Total & $\dbinom{\ell+1}{4}$ & $q=\ell^2-2\ell-3$ & $\dev=\ell^3-2\ell^2-\ell+2$		
	\end{tabular}
	
	\caption{Computing $\dev(G)$ where $G$ is a $\sigma$-cycle}\label{tbl:Cycle}
\end{table}

\section{Free vertices and a conjecture}\label{sec:FreeVertex}
In this final section we discuss \textit{free vertices} of a multiplicity on a graphic arrangement and present a conjecture on the structure of free multiplicities on braid arrangements.  


\begin{defn}
Suppose $G$ is a graph. A vertex $v_i\in V_G$ is a \textit{simplicial vertex} if the sub-graph of $G$ induced by $v_i$ and its neighbors is a complete graph.  Given a multi-arrangement $(\A_G,\m)$ and the corresponding edge-labeled graph $(G,\m)$, a vertex $v_i$ is a \textit{free vertex} of $(G,\m)$  if it is a simplicial vertex and for every triangle with vertices $v_i,v_j,v_k$ we have $m_{ij}+m_{ik}\le m_{jk}+1$.
\end{defn}

\begin{thm}\label{thm:FreeVertex}
Suppose $G$ is a graph, $v_i$ is a free vertex of $(G,\m)$, and $G'$ is the induced sub-graph on the vertex set $V_G\setminus \{v_i\}$.  Then $(\A_G,\m)$ is free if and only if $(\A_{G'},\m_{G'})$ is free.
\end{thm}
\begin{proof}[Proof of Theorem~\ref{thm:FreeVertex}]
We use a result whose proof we omit since it is virtually identical to the proof of~\cite[Theorem~5.10]{EulerMult}.  Recall that a flat $X\in L(\A)$ is called \textit{modular} if $X+Y\in L(\A)$ for every $Y\in L(\A)$, where $X+Y$ is the linear span of $X,Y$ considered as linear sub-spaces of $V=\kk^\ell$.
	
\begin{thm}\label{thm:SupersolvableMultiarrangements}
Suppose $(\A,\m)$ is a central multi-arrangement of rank $\ell\ge 3$ and $X$ is a modular flat of rank $\ell-1$.  Suppose $(\A_X,\m_X)$ is free with exponents $(d_1,\ldots,d_{\ell-1},0)$ and for all $H\in\A\setminus\A_X$ and $H'\in\A_X$, set $Y:=H\cap H'$.  If one of the following two conditions is satisfied:
\begin{enumerate}
\item $\A_Y=H\cup H'$ or
\item $\m(H')\ge \sum\limits_{H\in\A\setminus\A'} \m(H)-1$.
\end{enumerate}
Then $(\A,\m)$ is free with exponents $(d_1,\ldots,d_{\ell-1},|\m|-|\m'|)$.
\end{thm}
	
Now suppose $G$ is a graph on $\ell+1$ vertices $\{v_0,\ldots,v_\ell\}$ and $\A_G$ is the associated graphic arrangement.  Further suppose that $v_i$ is a free vertex of $(G,\m)$, and $G'$ is the induced sub-graph on the vertex set $V_G\setminus \{v_i\}$.  Set $\m'=\m|_{G'}$.  By Proposition~\ref{prop:FreeClosed}, if $(\A_{G'},\m')$ is not free, then neither is $(\A_G,\m)$.
	
Suppose now that $(\A_{G'},\m')$ is free.  We show that $(\A_G,\m)$ is free using Theorem~\ref{thm:SupersolvableMultiarrangements}.  Write $H_{ij}=V(x_i-x_j)$.  Since $v_i$ is a simplicial vertex of $G$, the flat $X=\cap_{v_j,v_k\neq v_i} H_{jk}$ is modular and has rank $\ell-1$.  The sub-arrangement $(\A_G)_X$ is the graphic arrangement $\A_{G'}$.  Suppose $H=H_{ij}\in\A_G\setminus\A_{G'}$, $H'=H_{st}\in\A_{G'}$, and set $Y=H_{ij}\cap H_{st}$.  If $\{s,t\}\cap\{i,j\}=\emptyset$, then $\A_Y=H_{ij}\cup H_{st}$.  Otherwise, suppose $s=j$.  Since $v_i$ is a simplicial vertex, $\{i,t\}\in E_G$, so $\A_Y=H_{ij}\cup H_{it}\cup H_{jt}$.  Since $v_i$ is a free vertex, $m_{ij}+m_{it}\le m_{jt}+1$, which is condition (2) from Theorem~\ref{thm:SupersolvableMultiarrangements}.  Hence $(\A_G,\m)$ is free by Theorem~\ref{thm:SupersolvableMultiarrangements}.
\end{proof}

\begin{remark}
Theorem~\ref{thm:FreeVertex} can also be proved using homological techniques from~\cite{GSplinesGraphicArrangements}.
\end{remark}

We use Theorem~\ref{thm:FreeVertex} to inductively construct two types of free multiplicities.  Given a graph $G$, an \textit{elimination ordering} is an ordering $v_0,\ldots,v_\ell$ of the vertices $V_G$ so that $v_i$ is a simplicial vertex of the induced sub-graph on $v_0,\ldots,v_i$ for every $i=1,\ldots,\ell$.  It is known that $V_G$ admits an elimination ordering if and only if $G$ is chordal~\cite{Dirac}.

\begin{cor}\label{cor:chordalFree}
Suppose $(G,\m)$ is an edge-labeled chordal graph with elimination ordering $v_0,\ldots,v_\ell$ satisfying that $v_i$ is a free vertex of the induced sub-graph on $\{v_0,\ldots,v_i\}$ for every $i\ge 2$.  Then $(\A_G,\m)$ is free.
\end{cor}

\begin{cor}\label{cor:FreeMultiplicitiesBraid}
Let $(\A_\ell,\m)$ be a multi-braid arrangement corresponding to the complete graph $K_{\ell+1}$ on $(\ell+1)$ vertices.  Suppose that $K_{\ell+1}$ admits an ordering $\{v_0,\ldots,v_\ell\}$ so that:
\begin{enumerate}
\item For some integer $0\le k\le \ell$, the induced sub-graph $G'$ on $\{v_0,\ldots,v_k\}$ satisfies that $\m_{G'}$ is a free ANN multiplicity.
\item For $k+1\le i\le \ell$, $v_i$ is a free vertex of the induced graph on $\{v_0,\ldots,v_i\}$.
\end{enumerate}
Then $(A_\ell,\m)$ is free.
\end{cor}

We conjecture that all free multi-braid arrangements take the form of Corollary~\ref{cor:FreeMultiplicitiesBraid}.

\begin{conj}\label{conj:1}
The multi-braid arrangement $(A_\ell,\m)$ is free if and only if it is one of the multi-braid arrangements constructed in Corollary~\ref{cor:FreeMultiplicitiesBraid}.  Equivalently, by Theorem~\ref{thm:FreeVertex}, if $(A_\ell,\m)$ is free then either $\m$ is a free ANN multiplicity or $\m$ has a free vertex.  Using Theorem~\ref{thm:1}, this is equivalent to the following statement: if $\m$ is a free multiplicity which is not in the balanced cone of multiplicities, then $\m$ has a free vertex.
\end{conj}

\begin{remark}
Conjecture~\ref{conj:1} is proved for the $A_3$ braid arrangement in~\cite{A3MultiBraid}.  Using Macaulay2~\cite{M2}, we have verified Conjecture~\ref{conj:1} for many multiplicities on the $A_4$ arrangement.
\end{remark}

\section{Acknowledgements}
I would especially like to thank Jeff Mermin, Chris Francisco, and Jay Schweig for their collaboration on the analysis of free multiplicities on the $A_3$ braid arrangement.  The current work would not be possible without their help.  I am very grateful to Takuro Abe for freely corresponding and offering many suggestions.  Computations in Macaulay2~\cite{M2} were indispensable for this project.

\bibliography{GraphBib,SplinesBib}{}
\bibliographystyle{plain}

\newpage

\appendix

\section{Computations for mountains and hills}\label{app:1}

\renewcommand{\arraystretch}{3}
\begin{longtable}{c|c|c|c}
		\multicolumn{4}{c}{
			$\sigma$-mountain on $(\ell+1)$ vertices
		}
		\\
		\hline
		Type of subgraph & Count & $q_U$ & $\dev(G_U)$ \\
		
		\hline
		
		\raisebox{-.5\height}{
			\begin{tikzpicture}[scale=.8]
			\tikzstyle{dot}=[circle,fill=black,inner sep=1 pt];
			
			\node[dot] (1) {};
			\node[dot] (2)[left of=1]{};
			\node[dot] (3)[below of=2]{};
			\node[dot] (4)[below of=1]{};

			\end{tikzpicture}
		} 
		& 
		$\dbinom{\ell-3}{4}$ & $0$ & $0$ \\[10 pt]
		
		\hline
		
		\raisebox{-.5\height}{
			\begin{tikzpicture}[scale=.8]
			\tikzstyle{dot}=[circle,fill=black,inner sep=1 pt];
			
			\node[dot] (1) {};
			\node[dot] (2)[left of=1]{};
			\node[dot] (3)[below of=2]{};
			\node[dot] (4)[below of=1]{};
			
			\draw[double distance=1 pt] (1)--(2);
			\end{tikzpicture}
		}
		&
		$3\dbinom{\ell-3}{3}$ & $2$ & $2$ \\[10 pt]
		
		\hline
		
		\raisebox{-.5\height}{
			\begin{tikzpicture}[scale=.8]
			\tikzstyle{dot}=[circle,fill=black,inner sep=1 pt];
			
			\node[dot] (1) {};
			\node[dot] (2)[left of=1]{};
			\node[dot] (3)[below of=2]{};
			\node[dot] (4)[below of=1]{};
			
			\draw[double distance=1 pt] (1)--(2)--(3);
			\end{tikzpicture}
		}
		&
		$2\dbinom{\ell-3}{2}$ & $2$ & $2$ \\[10 pt]
		
		\hline
		
		\raisebox{-.5\height}{
			\begin{tikzpicture}[scale=.8]
			\tikzstyle{dot}=[circle,fill=black,inner sep=1 pt];
			
			\node[dot] (1) {};
			\node[dot] (2)[left of=1]{};
			\node[dot] (3)[below of=2]{};
			\node[dot] (4)[below of=1]{};
			
			\draw[double distance=1 pt] (1)--(2) (3)--(4);
			\end{tikzpicture}
		}
		&
		$2\ell-9+\dbinom{\ell-5}{2}$ & $4$ & $8$ \\[10 pt]
		
		\hline
		
		\raisebox{-.5\height}{
			\begin{tikzpicture}[scale=.8]
			\tikzstyle{dot}=[circle,fill=black,inner sep=1 pt];
			
			\node[dot] (1) {};
			\node[dot] (2)[left of=1]{};
			\node[dot] (3)[below of=2]{};
			\node[dot] (4)[below of=1]{};
			
			\draw[double distance=1 pt] (1)--(2)--(3)--(4);
			\end{tikzpicture}
		}
		
		&
		$(\ell-3)$ & $2$ & $6$ \\[10 pt]
		
		\hline
		
		\raisebox{-.5\height}{
			\begin{tikzpicture}[scale=.8]
			\tikzstyle{dot}=[circle,fill=black,inner sep=1 pt];
			
			\node[dot] (1) {};
			\node[dot] (2)[left of=1]{};
			\node[dot] (3)[below of=2]{};
			\node[dot] (4)[below of=1]{};
			
			\draw (1)--(2);
			\end{tikzpicture}
		}
		&
		$\ell-4$ & $2$ & $2$\\[10 pt]
		
		\hline
		
		\raisebox{-.5\height}{
			\begin{tikzpicture}[scale=.8]
			\tikzstyle{dot}=[circle,fill=black,inner sep=1 pt];
			
			\node[dot] (1) {};
			\node[dot] (2)[left of=1]{};
			\node[dot] (3)[below of=2]{};
			\node[dot] (4)[below of=1]{};
			
			\draw (1)--(2);
			\draw[double distance=1 pt] (2)--(3);
			\end{tikzpicture}
		}
		&
		$2$ & $2$ & $6$\\[10 pt]
		
		\hline
		
		\raisebox{-.5\height}{
			\begin{tikzpicture}[scale=.8]
			\tikzstyle{dot}=[circle,fill=black,inner sep=1 pt];
			
			\node[dot] (1) {};
			\node[dot] (2)[left of=1]{};
			\node[dot] (3)[below of=2]{};
			\node[dot] (4)[below of=1]{};
			
			\draw (1)--(2)--(3);
			\end{tikzpicture}
		}
		&
		$2\dbinom{\ell-4}{2}$ & $2$ & $2$\\[10 pt]
		
		\hline
		
		\raisebox{-.5\height}{
			\begin{tikzpicture}[scale=.8]
			\tikzstyle{dot}=[circle,fill=black,inner sep=1 pt];
			
			\node[dot] (1) {};
			\node[dot] (2)[left of=1]{};
			\node[dot] (3)[below of=2]{};
			\node[dot] (4)[below of=1]{};
			
			\draw (1)--(2)--(3);
			\draw[double distance=1 pt] (1)--(3);
			\end{tikzpicture}
		}
		&
		$2(\ell-4)$ & $4$ & $8$\\[10 pt]
		
		\hline
		
		\raisebox{-.5\height}{
			\begin{tikzpicture}[scale=.8]
			\tikzstyle{dot}=[circle,fill=black,inner sep=1 pt];
			
			\node[dot] (1) {};
			\node[dot] (2)[left of=1]{};
			\node[dot] (3)[below of=2]{};
			\node[dot] (4)[below of=1]{};
			
			\draw (1)--(2)--(3);
			\draw[double distance=1 pt] (3)--(4);
			\end{tikzpicture}
		}
		&
		$2(\ell-4)$ & $2$ & $2$\\[10 pt]
		
		\hline
		
		\raisebox{-.5\height}{
			\begin{tikzpicture}[scale=.8]
			\tikzstyle{dot}=[circle,fill=black,inner sep=1 pt];
			
			\node[dot] (1) {};
			\node[dot] (2)[left of=1]{};
			\node[dot] (3)[below of=2]{};
			\node[dot] (4)[below of=1]{};
			
			\draw (1)--(2)--(3);
			\draw[double distance=1 pt] (1)--(3)--(4);
			\end{tikzpicture}
		}
		&
		$2$ & $2$ & $6$\\[10 pt]
		
		\hline
		
		\raisebox{-.5\height}{
			\begin{tikzpicture}[scale=.8]
			\tikzstyle{dot}=[circle,fill=black,inner sep=1 pt];
			
			\node[dot] (1) {};
			\node[dot] (2)[left of=1]{};
			\node[dot] (3)[below of=2]{};
			\node[dot] (4)[below of=1]{};
			
			\draw (1)--(2)--(3) (2)--(4);
			\end{tikzpicture}
		}
		&
		$\dbinom{\ell-4}{3}$ & $0$ & $0$\\[10 pt]
		
		\hline
		
		\raisebox{-.5\height}{
			\begin{tikzpicture}[scale=.8]
			\tikzstyle{dot}=[circle,fill=black,inner sep=1 pt];
			
			\node[dot] (1) {};
			\node[dot] (2)[left of=1]{};
			\node[dot] (3)[below of=2]{};
			\node[dot] (4)[below of=1]{};
			
			\draw (1)--(2)--(3) (2)--(4);
			\draw[double distance=1 pt] (3)--(4);
			\end{tikzpicture}
		}
		&
		$2\dbinom{\ell-4}{2}$ & $2$ & $2$\\[10 pt]
		
		\hline
		
		\raisebox{-.5\height}{
			\begin{tikzpicture}[scale=.8]
			\tikzstyle{dot}=[circle,fill=black,inner sep=1 pt];
			
			\node[dot] (1) {};
			\node[dot] (2)[left of=1]{};
			\node[dot] (3)[below of=2]{};
			\node[dot] (4)[below of=1]{};
			
			\draw (1)--(2)--(3) (2)--(4);
			\draw[double distance=1 pt] (3)--(4)--(1);
			\end{tikzpicture}
		}
		&
		$\ell-4$ & $2$ & $2$\\[10 pt]
		
		\hline
		
		Total & $\dbinom{\ell+1}{4}$ & $q_G=\ell^2-2\ell-3$ & $\dev=\ell^3-2\ell^2-\ell+2$\\	

	\caption{Computing $\dev(G)$ where $G$ is a $\sigma$-mountain}\label{tbl:mountain}

	\end{longtable}
	
	\renewcommand{\arraystretch}{3}
	\begin{longtable}{c|c|c|c}
		\multicolumn{4}{c}{
			$\sigma$-hill on $(\ell+1)$ vertices
		}
		\\
		\hline
		Type of subgraph & Count & $q_U$ & $\dev(G_U)$ \\
		
		\hline
		
		\raisebox{-.5\height}{
			\begin{tikzpicture}[scale=.8]
			\tikzstyle{dot}=[circle,fill=black,inner sep=1 pt];
			
			\node[dot] (1) {};
			\node[dot] (2)[left of=1]{};
			\node[dot] (3)[below of=2]{};
			\node[dot] (4)[below of=1]{};

			\end{tikzpicture}
		} 
		& 
		$\dbinom{\ell-4}{4}$ & $0$ & $0$ \\[10 pt]
		
		\hline
		
		\raisebox{-.5\height}{
			\begin{tikzpicture}[scale=.8]
			\tikzstyle{dot}=[circle,fill=black,inner sep=1 pt];
			
			\node[dot] (1) {};
			\node[dot] (2)[left of=1]{};
			\node[dot] (3)[below of=2]{};
			\node[dot] (4)[below of=1]{};
			
			\draw[double distance=1 pt] (1)--(2);
			\end{tikzpicture}
		}
		&
		$3\dbinom{\ell-4}{3}$ & $2$ & $2$ \\[10 pt]
		
		\hline
		
		\raisebox{-.5\height}{
			\begin{tikzpicture}[scale=.8]
			\tikzstyle{dot}=[circle,fill=black,inner sep=1 pt];
			
			\node[dot] (1) {};
			\node[dot] (2)[left of=1]{};
			\node[dot] (3)[below of=2]{};
			\node[dot] (4)[below of=1]{};
			
			\draw[double distance=1 pt] (1)--(2)--(3);
			\end{tikzpicture}
		}
		&
		$2\dbinom{\ell-4}{2}$ & $2$ & $2$ \\[10 pt]
		
		\hline
		
		\raisebox{-.5\height}{
			\begin{tikzpicture}[scale=.8]
			\tikzstyle{dot}=[circle,fill=black,inner sep=1 pt];
			
			\node[dot] (1) {};
			\node[dot] (2)[left of=1]{};
			\node[dot] (3)[below of=2]{};
			\node[dot] (4)[below of=1]{};
			
			\draw[double distance=1 pt] (1)--(2) (3)--(4);
			\end{tikzpicture}
		}
		&
		$2\ell-11+\dbinom{\ell-6}{2}$ & $4$ & $8$ \\[10 pt]
		
		\hline
		
		\raisebox{-.5\height}{
			\begin{tikzpicture}[scale=.8]
			\tikzstyle{dot}=[circle,fill=black,inner sep=1 pt];
			
			\node[dot] (1) {};
			\node[dot] (2)[left of=1]{};
			\node[dot] (3)[below of=2]{};
			\node[dot] (4)[below of=1]{};
			
			\draw[double distance=1 pt] (1)--(2)--(3)--(4);
			\end{tikzpicture}
		}
		
		&
		$\ell-4$ & $2$ & $6$ \\[10 pt]
		
		\hline
		
		\raisebox{-.5\height}{
			\begin{tikzpicture}[scale=.8]
			\tikzstyle{dot}=[circle,fill=black,inner sep=1 pt];
			
			\node[dot] (1) {};
			\node[dot] (2)[left of=1]{};
			\node[dot] (3)[below of=2]{};
			\node[dot] (4)[below of=1]{};
			
			\draw (1)--(2)--(3);
			\end{tikzpicture}
		}
		&
		$2\dbinom{\ell-4}{2}$ & $2$ & $2$\\[10 pt]
		
		\hline
		
		\raisebox{-.5\height}{
			\begin{tikzpicture}[scale=.8]
			\tikzstyle{dot}=[circle,fill=black,inner sep=1 pt];
			
			\node[dot] (1) {};
			\node[dot] (2)[left of=1]{};
			\node[dot] (3)[below of=2]{};
			\node[dot] (4)[below of=1]{};
			
			\draw (1)--(2)--(3);
			\draw[double distance=1 pt] (1)--(3);
			\end{tikzpicture}
		}
		&
		$2(\ell-4)$ & $4$ & $8$\\[10 pt]
		
		\hline
		
		\raisebox{-.5\height}{
			\begin{tikzpicture}[scale=.8]
			\tikzstyle{dot}=[circle,fill=black,inner sep=1 pt];
			
			\node[dot] (1) {};
			\node[dot] (2)[left of=1]{};
			\node[dot] (3)[below of=2]{};
			\node[dot] (4)[below of=1]{};
			
			\draw (1)--(2)--(3);
			\draw[double distance=1 pt] (3)--(4);
			\end{tikzpicture}
		}
		&
		$2(\ell-4)$ & $2$ & $2$\\[10 pt]
		
		\hline
		
		\raisebox{-.5\height}{
			\begin{tikzpicture}[scale=.8]
			\tikzstyle{dot}=[circle,fill=black,inner sep=1 pt];
			
			\node[dot] (1) {};
			\node[dot] (2)[left of=1]{};
			\node[dot] (3)[below of=2]{};
			\node[dot] (4)[below of=1]{};
			
			\draw (1)--(2)--(3);
			\draw[double distance=1 pt] (1)--(3)--(4);
			\end{tikzpicture}
		}
		&
		$2$ & $2$ & $6$\\[10 pt]
		
		\hline
		
		\raisebox{-.5\height}{
			\begin{tikzpicture}[scale=.8]
			\tikzstyle{dot}=[circle,fill=black,inner sep=1 pt];
			
			\node[dot] (1) {};
			\node[dot] (2)[left of=1]{};
			\node[dot] (3)[below of=2]{};
			\node[dot] (4)[below of=1]{};
			
			\draw (1)--(2)--(3) (2)--(4);
			\end{tikzpicture}
		}
		&
		$2\dbinom{\ell-4}{3}$ & $0$ & $0$\\[10 pt]
		
		\hline
		
		\raisebox{-.5\height}{
			\begin{tikzpicture}[scale=.8]
			\tikzstyle{dot}=[circle,fill=black,inner sep=1 pt];
			
			\node[dot] (1) {};
			\node[dot] (2)[left of=1]{};
			\node[dot] (3)[below of=2]{};
			\node[dot] (4)[below of=1]{};
			
			\draw (1)--(2)--(3) (2)--(4);
			\draw[double distance=1 pt] (3)--(4);
			\end{tikzpicture}
		}
		&
		$4\dbinom{\ell-4}{2}$ & $2$ & $2$\\[10 pt]
		
		\hline
		
		\raisebox{-.5\height}{
			\begin{tikzpicture}[scale=.8]
			\tikzstyle{dot}=[circle,fill=black,inner sep=1 pt];
			
			\node[dot] (1) {};
			\node[dot] (2)[left of=1]{};
			\node[dot] (3)[below of=2]{};
			\node[dot] (4)[below of=1]{};
			
			\draw (1)--(2)--(3) (2)--(4);
			\draw[double distance=1 pt] (3)--(4)--(1);
			\end{tikzpicture}
		}
		&
		$2(\ell-4)$ & $2$ & $2$\\[10 pt]
		
		\hline
		
		\raisebox{-.5\height}{
			\begin{tikzpicture}[scale=.8]
			\tikzstyle{dot}=[circle,fill=black,inner sep=1 pt];
			
			\node[dot] (1) {};
			\node[dot] (2)[left of=1]{};
			\node[dot] (3)[below of=2]{};
			\node[dot] (4)[below of=1]{};
			
			\draw (1)--(2)--(3)--(4);
			\end{tikzpicture}
		}
		&
		$1$ & $2$ & $6$\\[10 pt]
		
		\hline
		
		\raisebox{-.5\height}{
			\begin{tikzpicture}[scale=.8]
			\tikzstyle{dot}=[circle,fill=black,inner sep=1 pt];
			
			\node[dot] (1) {};
			\node[dot] (2)[left of=1]{};
			\node[dot] (3)[below of=2]{};
			\node[dot] (4)[below of=1]{};
			
			\draw (1)--(2)--(3)--(1) (2)--(4);
			\end{tikzpicture}
		}
		&
		$2(\ell-4)$ & $2$ & $2$\\[10 pt]
		
		\hline
		
		\raisebox{-.5\height}{
			\begin{tikzpicture}[scale=.8]
			\tikzstyle{dot}=[circle,fill=black,inner sep=1 pt];
			
			\node[dot] (1) {};
			\node[dot] (2)[left of=1]{};
			\node[dot] (3)[below of=2]{};
			\node[dot] (4)[below of=1]{};
			
			\draw (1)--(2)--(3)--(1) (2)--(4);
			\draw[double distance=1 pt] (3)--(4);
			\end{tikzpicture}
		}
		&
		$2$ & $2$ & $6$\\[10 pt]
		
		\hline
		
		\raisebox{-.5\height}{
			\begin{tikzpicture}[scale=.8]
			\tikzstyle{dot}=[circle,fill=black,inner sep=1 pt];
			
			\node[dot] (1) {};
			\node[dot] (2)[left of=1]{};
			\node[dot] (3)[below of=2]{};
			\node[dot] (4)[below of=1]{};
			
			\draw (1)--(2)--(3)--(1) (2)--(4)--(3);
			\end{tikzpicture}
		}
		&
		$\dbinom{\ell-4}{2}$ & $2$ & $2$\\[10 pt]
		
		\hline
		
		\raisebox{-.5\height}{
			\begin{tikzpicture}[scale=.8]
			\tikzstyle{dot}=[circle,fill=black,inner sep=1 pt];
			
			\node[dot] (1) {};
			\node[dot] (2)[left of=1]{};
			\node[dot] (3)[below of=2]{};
			\node[dot] (4)[below of=1]{};
			
			\draw (1)--(2)--(3)--(1) (2)--(4)--(3);
			\draw[double distance=1 pt] (4)--(1);
			\end{tikzpicture}
		}
		&
		$\ell-4$ & $4$ & $8$\\[10 pt]
		
		\hline
		
		Total & $\dbinom{\ell+1}{4}$ & $q_G=\ell^2-2\ell-3$ & $\dev=\ell^3-2\ell^2-\ell+2$\\	
		
		\caption{Computing $\dev(G)$ where $G$ is a $\sigma$-hill}\label{tbl:hill}
		
	\end{longtable}
	
\end{document}